\documentclass[11pt,a4paper,notitlepage]{article}
\usepackage[latin1]{inputenc}
\usepackage[T1]{fontenc}
\usepackage{amsmath}
\usepackage{amssymb}
\usepackage{amsfonts}
\usepackage{amsthm}
\usepackage[pdftex]{graphicx}
\usepackage{enumerate}
\usepackage{afterpage}
\usepackage{hyperref}
  \hypersetup{bookmarksopen=true}

\newtheorem{thm}{Theorem}
\newtheorem{lem}[thm]{Lemma}

\newtheorem{rem}[thm]{Remark}
\newtheorem*{rbib*}{Note}

\newcommand{\abs}[1]{\left|#1\right|}
\newcommand{\norm}[1]{\left\|#1\right\|}

\begin{document}
\title{Stepsize control for Newton's method in the presence of singularities}
\author{Michael Kratzer\footnote{Zentrum Mathematik, Technische Universität München, Boltzmannstraße 3, 85748 Garching bei München, Germany (\texttt{mkratzer@ma.tum.de})}}
\maketitle
\begin{abstract}
Singularities in the Jacobian matrix are an obstacle to Newton's method. We show that stepsize controls suggested by Deuflhard(\cite{deu74}) and Steinhoff(\cite{ste11}) can be used to detect and rapidly converge to such singularities.
\end{abstract}

\section{Introduction}
We are concerned with the numerical solution of nonlinear equations $F(x)=0$, $F: \mathbb{R}^d \rightarrow \mathbb{R}^d$.
While the undamped version of Newton's method can exhibit chaotic behavior outside the neighborhoods of roots, damped, globalized methods (see e.g. \cite{deu04} for an overview) can reliably find solutions by approximately tracing the Newton path
\begin{equation}\label{newtonPath}
\dot x = -DF^{-1}F.
\end{equation}
As \eqref{newtonPath} implies $\dot F = -F$, such methods will eventually converge to a solution unless the trajectory is interrupted by a singular Jacobian $DF$. It would therefore be useful if a method could efficiently diagnose this condition instead of merely slowing down and eventually terminating due to the stepsize becoming too small.

This paper shows that the affine covariant stepsize control introduced by Deuflhard (\cite{deu74}) and a modification based on Projected Natural Level Functions proposed by Steinhoff (\cite{ste11}) both posses such an ability. To this end we will present a singularity indicator based on \cite{gr84} (cf. Appendix \ref{grCon} for the connection) and derive from it stepsize controls that turn out to coincide with those mentioned above.
These schemes are then analyzed and we conclude with some numerical experiments.

\section{The singularity indicator}

\subsection{Definition}
Let $F: \mathbb{R}^d \rightarrow \mathbb{R}^d$ be sufficiently smooth.
We define for $v\neq0$
\begin{equation} \label{def_g}
g(x, v) := \begin{cases}
\norm{DF(x)^{-1}\frac{F(x)}{\norm{F(x)}}}^{-1} & \text{if $DF(x)$ is not singular and $F(x) \neq 0$}, \\
0 & \text{if $F(x)\notin \mathcal{R}(DF(x))$}, \\
\lim\limits_{\epsilon \searrow 0} g(x+\epsilon \frac{v}{\norm{v}})& \text{otherwise}.
\end{cases}
\end{equation}
Note that $g$ is not properly defined everywhere as the limit may fail to exist. We will see that this does not matter for our purposes. 
In most cases $g$ will not depend on $v$ and we will speak only of $g(x)$.

The motivation of this definition is that, informally by division through zero, $\norm{DF(x)^{-1}\frac{F(x)}{\norm{F(x)}}}$ should be infinite and $g$ zero if $DF(x)$ is singular. 

\begin{rbib*}
The indicator $g$ is a special case of the indicator proposed by Griewank and Reddien in \cite{gr84}. However, the indicator of Griewank of Reddien is only defined if the rank-deficiency is at most one. Details can be found in Appendix \ref{grCon}.
\end{rbib*}

\section{Exact stepsize control}
We begin by deriving a stepsize control from the first case of \eqref{def_g} and deal with the other cases later.

For the following calculations it is convenient to introduce some abbreviations.
Let $J(x):=DF(x)$, $H(x):=D^2F(x)$, $R(x) := \frac{F(x)}{\norm{F(x)}}$ and $T(x) := \frac{DF(x)^{-1}F(x)}{\norm{DF(x)^{-1}F(x)}}$.

Using $g=f\circ h$ with $f(x):=\norm{x}^{-1}$ and $h(x):=J(x)^{-1}R(x)$ we then compute
\begin{equation*}
\begin{aligned}
Df(x)(v)& = -\norm{x}^{-2}\left<\frac{x}{\norm{x}},v\right> = -\norm{x}^{-3}\left< x,v \right> \\
D^2f(x)(v,u) & = 3\norm{x}^{-5}\left< x,v \right> \left< x,u \right> - \norm{x}^{-3} \left< u, v \right>
\end{aligned}
\end{equation*}
and, suppressing $x$,
\begin{align}
\nonumber Dh(v) & = -J^{-1}H(v,\cdot)J^{-1}R + J^{-1}DR(v) \\
\nonumber D^2h(v,u) & = J^{-1}H(u,\cdot)J^{-1}H(v,\cdot)J^{-1}R - J^{-1}DH(v,u,\cdot)J^{-1}R 
\\ \nonumber & \quad + J^{-1}H(v,\cdot)J^{-1}H(u,\cdot)J^{-1}R - J^{-1}H(u,\cdot)J^{-1}DR(v) 
\\ & \quad  - J^{-1}H(v,\cdot)J^{-1}DR(u) + J^{-1}D^2R(v,u)  \label{d2h} \\
Dg(v) & = \norm{J^{-1}R}^{-2}\left< T, J^{-1}H(v,\cdot)J^{-1}R - J^{-1}DR(v) \right> \label{dg}
\end{align}

Along the Newton direction $\Delta x := - J^{-1}F$ we have $J(\Delta x) = -F$ and hence $DR(\Delta x)=0$. \eqref{dg} then yields
\begin{align*}
Dg(\Delta x) & = \norm{\frac{\Delta x}{\norm{F}}}^{-2}\left< -\frac{\Delta x}{\norm{\Delta x}}, J^{-1}H(\Delta x,\cdot)\frac{\Delta x}{\norm{F}} \right> \\
& = \frac{\norm{F}}{\norm{\Delta x}} \left< -\frac{\Delta x}{\norm{\Delta x}}, J^{-1}H\left(\Delta x, \frac{\Delta x}{\norm{\Delta x}}\right) \right>.
\end{align*}

As we do not want to go past points with singular $DF$, this leads for to the stepsize restriction
\begin{equation}\tag{ES}\label{exstep}
\lambda^{(k)} \overset{!}{\leq} \frac{g(x^{(k)})}{Dg(x^{(k)})(\Delta x^{(k)})} =
\left< \frac{\Delta x^{(k)}}{\norm{\Delta x^{(k)}}}, DF(x^{(k)})^{-1}D^2 F(x^{(k)})\left(\Delta x^{(k)},\frac{\Delta x^{(k)}}{\norm{\Delta x^{(k)}}}\right) \right>^{-1},
\end{equation}
whenever $\left< \frac{\Delta x^{(k)}}{\norm{\Delta x^{(k)}}}, DF(x^{(k)})^{-1}D^2 F(x^{(k)})\left(\Delta x^{(k)},\frac{\Delta x^{(k)}}{\norm{\Delta x^{(k)}}}\right) \right> > 0$, for the damped Newton step
\begin{equation*}
x^{(k+1)} := x^{(k)} + \lambda^{(k)}\Delta x^{(k)}.
\end{equation*}

\begin{rbib*}
In his thesis \cite{ste11}, Steinhoff proposed the stepsize control
\begin{equation*}
\lambda^{(k)} = \max \left\{ \abs{\left< \frac{\Delta x^{(k)}}{\norm{\Delta x^{(k)}}}, DF(x^{(k)})^{-1}D^2 F(x^{(k)})\left(\Delta x^{(k)},\frac{\Delta x^{(k)}}{\norm{\Delta x^{(k)}}}\right) \right>}^{-1}, 1 \right\},
\end{equation*}
which obviously fulfills \eqref{exstep}.
\end{rbib*}

For the purposes of stepsize control it is sufficient to consider the behavior of $g$ along lines, i.e. of $g(x+\epsilon v, v)$ as $\epsilon$ varies. As mentioned before, $g$ is undefined in the interior of a line segment on which points with $F=0$ or singular $DF$ are dense. However, the Newton method should terminate upon encountering such a segment and hence there is no need to move along it. Double roots, i.e. $x$ where $F(x)=0$ and $DF(x)$ is singular, are beyond the scope of this paper.

The remaining cases will be considered after introducing some necessary perturbation results.

\section{Perturbation Lemmas}
\begin{lem}\label{invBounds}
Let $L,M \in \mathcal{L}(X,Y)$ be bounded linear operators, such that $L$ is invertible and $\norm{M-L}<\frac{1}{\norm{L^{-1}}}$.

Then $M$ is invertible and
\begin{equation*}
\norm{M^{-1}} \leq \frac{\norm{L^{-1}}}{1-\norm{L^{-1}}\norm{L-M}}.
\end{equation*}
Furthermore,
\begin{equation*}
\norm{L^{-1}-M^{-1}}\leq \frac{\norm{L^{-1}}^2\norm{L-M}}{1-\norm{L^{-1}}\norm{L-M}}.
\end{equation*}
\end{lem}
\begin{proof}
\cite{ah09}
\end{proof}

\begin{lem}\label{pertInv}
Let $X$, $Y$ be Banach spaces and $A$, $B$: $X\rightarrow Y$ bounded linear operators. Assume that
\begin{enumerate}[(i)]
  \item $\dim \mathcal{N}(A) < \infty $
  \item $\mathcal{N}(A) \cap \mathcal{N}(B) = \left\{ 0 \right\}$,
  \item $Y = B\mathcal{N}(A) \oplus \mathcal{R}(A)$ with corresponding projectors $P$, $Id - P$.
\end{enumerate}
Let $\widehat{X}\subseteq X$ be a complement of $\mathcal{N}(A)$ with projectors $Id-Q$, $Q$
and $A^* : \mathcal{R}(A) \rightarrow \widehat{X}$, $B^* : B\mathcal{N}(A) \rightarrow \mathcal{N}(A)$ the inverses of the respective restrictions of $A$ and $B$. Then
\begin{itemize}
  \item $(A+\epsilon B)$ is invertible for sufficiently small $\epsilon > 0$
  \item $\forall b \in Y: \exists u,w \in X: (A+\epsilon B)(u + \frac{1}{\epsilon}w) \rightarrow b$ as $\epsilon \rightarrow 0$
  \item $\norm{(A^* (Id-P)+ \frac{1}{\epsilon}B^* P)-(A+\epsilon B)^{-1}} = O(1) \quad (\epsilon \rightarrow 0)$
\end{itemize}
\end{lem}
\begin{proof}
All limits are with respect to $\epsilon \rightarrow 0$. For the case of infinite dimensions, note that $B\mathcal{N}(A)$ has finite dimension. Hence $B^*$ is bounded, the complement $\mathcal{R}(A)$ is closed and so $A^*$ is also bounded.

Let $u := A^* (Id-P)b$ and $w:= B^* Pb \in \mathcal{N}(A)$. Then $(A+\epsilon B)(u + \frac{1}{\epsilon}w)$ $=$ $(Id-P)b + \epsilon BA^* (Id-P)b + Pb$ $=$ $b + O(\epsilon)$ $\rightarrow b$.

A norm on $X$ is given by $\norm{x}_\epsilon := \norm{(Id-Q)x + \epsilon Qx} \forall x\in X$. We also use $\norm{\cdot}_\epsilon$ to denote operator norms which are induced by $(X,\norm{\cdot}_\epsilon)$. We have
\[ \left( A^* (Id-P) + \frac{1}{\epsilon} B^* P \right)^{-1} = A(Id-Q) + \epsilon BQ = A + \epsilon B - \epsilon (Id-Q)B \]
and $\norm{A^* (Id-P) + \frac{1}{\epsilon} B^* P}_\epsilon = O(1)$ (recall that $\mathcal{R}(B^*P) = \mathcal{N}(A)$), $\norm{A(Id-Q) + \epsilon BQ}_\epsilon = O(1)$, $\norm{\epsilon(Id-Q)B}_\epsilon = O(\epsilon)$. With Lemma \ref{invBounds} it follows that $A+\epsilon B$ is invertible for sufficiently small $\epsilon > 0$ and that $\norm{\left( A+\epsilon B \right)^{-1} - \left(A^* (Id-P) + \frac{1}{\epsilon} B^* P\right) }_\epsilon = O(\epsilon)$. As the ordinary and $\epsilon-$norms are equivalent with constants of order $\epsilon$ resp. $\frac{1}{\epsilon}$, the final claim follows.
\end{proof}

\begin{rem}
The assumptions of Lemma \ref{pertInv} imply that $A$ is a Fredholm operator of index 0.
\end{rem}

\begin{rem}
If $\mathcal{N}(A) \cap \mathcal{N}(B) \neq \left\{ 0 \right\}$, $A+\epsilon B$ is not invertible for any $\epsilon$. In this case one can first restrict $A$ and $B$ to a complement of $\mathcal{N}(A) \cap \mathcal{N}(B)$ and then apply Lemma \ref{pertInv}.
\end{rem}

\begin{rem}
If $B\mathcal{N}(A) \oplus \mathcal{R}(A) \neq Y$, then $A+\epsilon B$ has a singular value $O(\epsilon^2)$ (cf. Appendix \ref{smoothSVD}) and hence its inverse (if it exists) grows at least with order $\epsilon^{-2}$.
\end{rem}

\begin{rem}\label{pertInvSimple}
The term $A^* (Id-P)$ in the approximate inverse of $A+\epsilon B$ matters only in the $\epsilon$-norm and we also have
\[  \norm{\frac{1}{\epsilon}B^* P-(A+\epsilon B)^{-1}} = O(1) \quad (\epsilon \rightarrow 0). \]
\end{rem}

\begin{lem}\label{InvOnR}
If $\widehat{X} = B^{-1}\mathcal{R}(A) := \left\{ x\in X : Bx\in\mathcal{R}(A) \right\}$ in the setting of Lemma~\ref{pertInv}, then, for all $y \in \mathcal{R}(A)$,
\begin{equation*}
\norm{\left(A^* - (A+\epsilon B)^{-1}\right)y} = O(\epsilon)\quad (\epsilon\rightarrow 0).
\end{equation*}
\end{lem}
\begin{proof}
Note that $\mathcal{R}\left((A+\epsilon B)A^*\right) \subseteq \mathcal{R}(A)$. Hence for the iteration $x_0 := A^*y$, $r_i := (A+\epsilon B)x_i - y$, $x_{i+1} := x_i + A^*r_i$  ($i=0,1,\ldots$) it holds that $r_i\in \mathcal{R}(A)$, $r_i = O(\epsilon^{i+1})$ and, for sufficiently small $\epsilon$, $x_i \xrightarrow[i\rightarrow\infty]{} x^* =: (A+\epsilon B)^{-1}y$ with $\norm{x_0-x^*} = O(\epsilon)$.
\end{proof}

\section{Case-by-case analysis}
\subsection{$DF$ singular, $F\notin \mathcal{R}(DF)$}
First, we briefly note that $g$ is continuous:

\begin{lem}
Let $DF(x_0)$ be singular and $F(x_0)\notin \mathcal{R}\left(DF(x_0)\right)$. Then $g$ is continuous at $x_0$.
\end{lem}
\begin{proof}
As $F(x_0)\notin \mathcal{R}\left(DF(x_0)\right)$, there exists $u\in \mathcal{R}(DF(x_0))^\perp$ such that $\left< u, R(x)\right> > c > 0$ in a neighborhood of $x_0$. If $v(x):=DF(x)^{-1}R(x)$ exists, then 
$c < \left< u, DF(x)v(x) \right>$ $= \left< u, (DF(x)-DF(x_0))v(x) \right>$ $= v(x)\cdot O(\norm{x-x_0})$.
Otherwise $g(x)=0$ and so in either case $g(x) = O(\norm{x-x_0})$ and hence $g(x)\rightarrow 0 = g(x_0)$ as $x\rightarrow x_0$.
\end{proof}

The interesting result for convergence is the following:

\begin{thm}\label{thmDirDeriv}
Let $J(x_0)$ be singular, $\mathcal{N}(J(x_0))\cap \mathcal{N}(H(x_0)(v,\cdot)) = \left\{ 0 \right\}$ and $F(x_0)\notin \mathcal{R}\left(J(x_0)\right)$.
Then for any $v\neq 0$, $\epsilon\mapsto g(x_0+\epsilon v, v)$ has a directional derivative at $\epsilon=0$ which is locally Lipschitz-continuous. 
\end{thm}
\begin{proof}
During this proof evaluations of any function at $x_0$ are denoted by the subscript $0$, otherwise function are evaluated at $x_0+\epsilon v$, which is suppressed, and we redefine $H:=H(v,\cdot)$. All limits are with respect to $\epsilon\searrow 0$.

$\mathcal{N}(J_0)\cap\mathcal{N}(H_0) = \left\{ 0 \right\}$ implies that $J$ is nonsingular for small $\epsilon\neq 0$ (cf. Section \ref{smoothSVD})
and hence $\frac{d}{d\epsilon}g$ is given by \eqref{dg}. It remains to show that the limit exists as $\epsilon\searrow 0$ and that $\frac{d}{d\epsilon^2}g$ is bounded.

Note that $R \rightarrow R_0$ with $DR$ bounded in a neighborhood of $x_0$ since $F_0 \neq 0$. Applying Lemma \ref{invBounds} and Remark \ref{pertInvSimple} with $A=J_0$, $B=H_0$ we have
\begin{align*}
J^{-1}R & = (J_0 + \epsilon H_0 + O(\epsilon^2))^{-1}R = (J_0 + \epsilon H_0)^{-1}R + O(1) \\
& = \frac{1}{\epsilon}H_0^*PR + O(1) = \frac{1}{\epsilon} H_0^*PR_0 + O(1).
\end{align*}
Similarly $J^{-1}HJ^{-1}R$ = $\frac{1}{\epsilon^2}H_0^*PH_0H_0^*PR + O\left(\frac{1}{\epsilon}\right)$ = $\frac{1}{\epsilon^2}H_0^*PR + O\left(\frac{1}{\epsilon}\right)$ and $J^{-1}DR =  O\left(\frac{1}{\epsilon}\right)$.

$H_0^*PR_0\neq 0$ because $F_0\notin \mathcal{R}(J_0)$ and hence $T\rightarrow  \frac{H_0^*PR_0}{\norm{H_0^*PR_0}} =: T_0$. Inserting into \eqref{dg} we  obtain
\begin{align*}
\frac{d}{d\epsilon}g & = \norm{J^{-1}R}^{-2}\left< T, J^{-1}H(v,\cdot)J^{-1}R - J^{-1}DR(v) \right> \\
&= \epsilon^2\left(\norm{H_0^*PR_0}^{-2}+O\left(\epsilon\right)\right)\left< T, \frac{1}{\epsilon^2} H_0^*PR_0 + O\left(\frac{1}{\epsilon}\right) \right> \\
& \rightarrow \frac{\left< T_0, H_0^*PR_0\right>}{\norm{H_0^*PR_0}^2}.
\end{align*}

For $h=J^{-1}R$ we have, using \eqref{d2h},
\begin{align*}
\frac{d}{d\epsilon^2}h &= 2 J^{-1}HJ^{-1}HJ^{-1}R + O\left( \frac{1}{\epsilon^2} \right)
=  \frac{2}{\epsilon^3} H_0^*PH_0H_0^*PH_0H_0^*PR + O\left( \frac{1}{\epsilon^2} \right) \\
& = \frac{2}{\epsilon^3} H_0^*PR + O\left( \frac{1}{\epsilon^2} \right) \\
\frac{d}{d\epsilon^2}g &= 3\norm{J^{-1}R}^{-5}\left< J^{-1}R, J^{-1}HJ^{-1}R+O\left(\frac{1}{\epsilon}\right) \right>^2
- \norm{J^{-1}R}^{-3}\norm{J^{-1}HJ^{-1}R+\left(\frac{1}{\epsilon}\right)}^2 \\
& \quad - \norm{J^{-1}R}^{-3}\left< J^{-1}R,  2 J^{-1}HJ^{-1}HJ^{-1}R + O\left( \frac{1}{\epsilon^2} \right) \right> \\
& = 3\norm{\frac{1}{\epsilon}H_0^*PR+O(1)}^{-5} \left< \frac{1}{\epsilon} H_0^*PR+O(1), \frac{1}{\epsilon^2} H_0^*PR+O\left(\frac{1}{\epsilon}\right) \right>^2 \\
& \quad - \norm{\frac{1}{\epsilon}H_0^*PR+O(1)}^{-3}\norm{\frac{1}{\epsilon^2} H_0^*PR+O\left(\frac{1}{\epsilon}\right)}^2 \\
& \quad - 2\norm{\frac{1}{\epsilon}H_0^*PR+O(1)}^{-3} \left< \frac{1}{\epsilon} H_0^*PR+O(1), \frac{1}{\epsilon^3} H_0^*PR+O\left(\frac{2}{\epsilon^2}\right)  \right> \\
& = O(1)
\end{align*}
\end{proof}

In general, full differentiability is not possible unless the gradient at a singularity of $J$ is 0.

\emph{Example} Let $F(x):=\begin{pmatrix}
x_1 & 0 \\ 0 & x_2
\end{pmatrix}x$. Then $g=0$ on both $\left\{ x: x_1=0 \right\}$ and $\left\{ x: x_2=0 \right\}$.

If rank-deficiency of $J$ is 1, $g$ only fails to be differentiable because the sign does not change when crossing the singular manifold. This could be remedied by using $\det(J)\cdot g$ instead of $g$ and differentiability could then be deduced from \cite{gr84}, cf. also Appendix \ref{grCon}.

\subsection{$DF$ singular, $0\neq F\in \mathcal{R}(DF)$}
Again the subscript $0$ denotes values at the singular point $x_0$. We have $F=F_0 + \epsilon J_0 v + O(\epsilon^2)$ and $R=R_0 + O(\epsilon)$. With Lemma \ref{InvOnR}, $\norm{J^{-1}R}$ = $\norm{J_0^*R_0 + O(\epsilon)}$ $< C$ 
and $\norm{D(J^{-1}R)}$ = $\norm{-J_0^*H_0J_0^*R_0 + O(1)}$ = $O(1)$, as $H_0J_0^*R_0\in\mathcal{R}(J_0)$. 
Hence $\abs{g/(Dg\cdot v)}$=$1/C/O(1)$, which means that the stepsize control does not detect such singularities. Since, in this particular case, they are no obstruction to solving $F=0$, one can justify not regarding this as a defect.

Note that $J_0^*$ and hence $g(x,v)$ depend on $v$ through the decomposition $X = H_0(v,\cdot)^{-1}\mathcal{R}(J_0) \oplus \mathcal{N}(J_0)$.

\subsection{$DF$ regular, $F=0$} \label{secBehaviorNearSolution}
Near a regular solution $x_0$, i.e. for $F(x_0)=0$, $J(x_0)$ invertible, we have $F = \epsilon J_0v + O(\epsilon^2)$, $R=\frac{J_0v}{\norm{J_0v}}+O(\epsilon)$, $J^{-1}R=\frac{v}{\norm{J_0v}}+O(\epsilon)$, $C > \norm{J^{-1}R}$ and $\norm{D(J^{-1}R)}$= $\norm{-J^{-1}HJ^{-1}R+DR}$=$O(1)$ as $x\rightarrow x^*$, uniformly for all $v$ and in some neighborhood of $x^*$. This implies $g > C^{-1}$, $Dg = O(\norm{x-x^*})$ and so the indicator will not erroneously predict a nearby singularity although $g$ is in general not continuous at $x_0$ and $g(x_0,v)$ depends on $v$.

\begin{figure}[h!tb]
    \begin{center}
		\begin{minipage}{0.9\linewidth}
		\includegraphics[width=\linewidth]{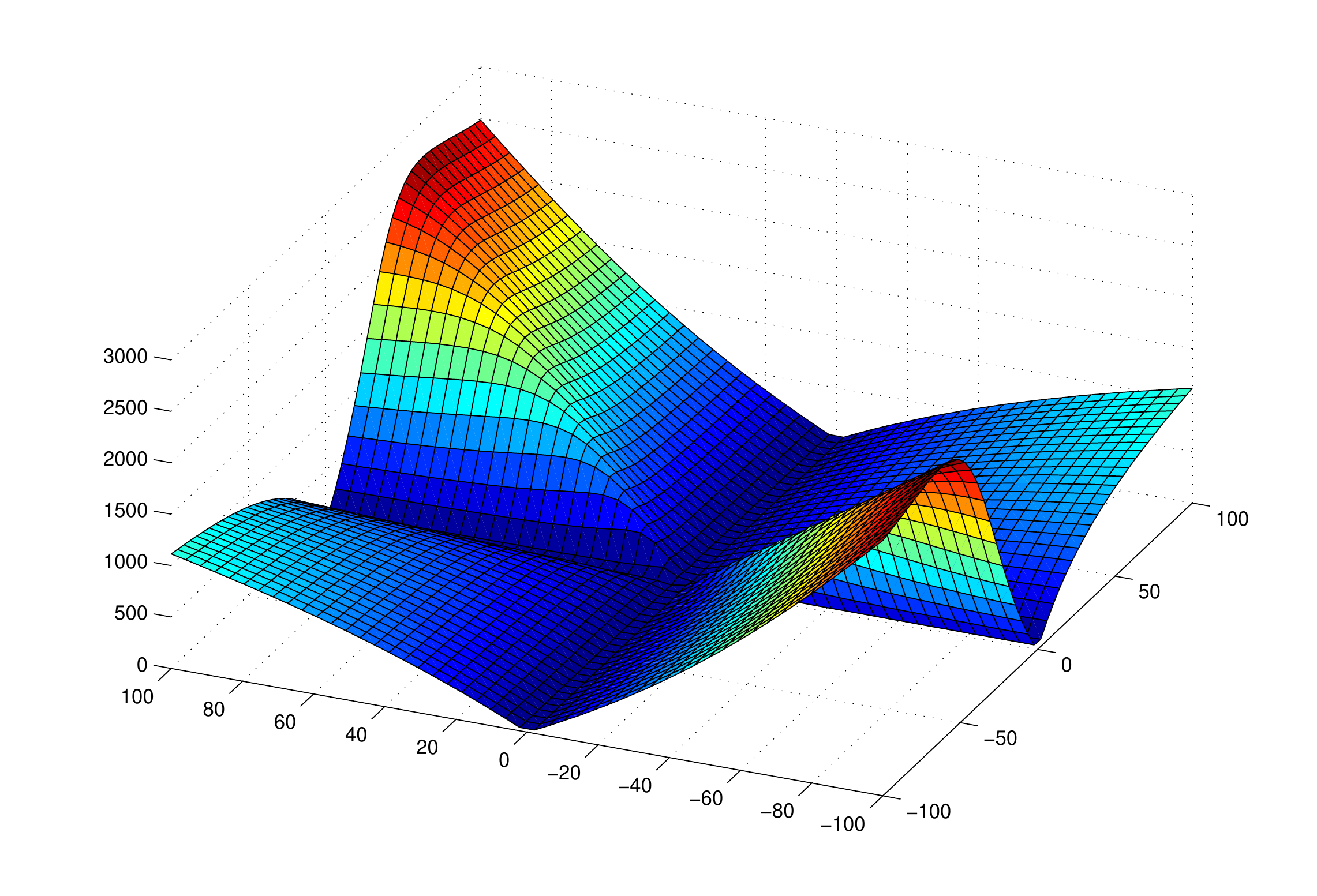}
		\caption[...]{$g$ at the crossing of two singular manifolds \label{figSC}}
		\end{minipage}
		\end{center}
\end{figure}
\section{Convergence to singularity}\label{convergence}
By the preceding remarks, the constraint is only active in the neighborhood of points with $F\notin \mathcal{R}(DF)$. If the direction $\Delta x / \norm{\Delta x}$ is assumed to be fixed, Theorem \ref{thmDirDeriv} ensures local quadratic convergence of $g$ to $0$.\footnote{Assuming, of course, that equality is chosen in \eqref{exstep} and $g/(Dg\cdot \Delta x) < 0$, i.e. the Newton direction actually points towards the singular manifold.} In well-behaved cases, in particular if $\mathcal{N}(DF)$ is one-dimensional and $\Delta x$ is therefore always approximately aligned to the direction uniquely given by $\mathcal{N}(DF)$, there is a common region of quadratic convergence for all $\Delta x^{(k)}$ which arise during the iteration. 

In general this is hard to guarantee and problems arise when singular manifolds intersect. Figure \ref{figSC} shows $g$ for the function 
\[ F:\mathbb{R}^2\rightarrow\mathbb{R}^2,x\mapsto -\left( x_1\begin{pmatrix}
5 & 10 \\ 2 & 4
\end{pmatrix} + x_2\begin{pmatrix}
4 & 2  \\ 6 & 3
\end{pmatrix} \right)x - 10^6\begin{pmatrix}
1.1 \\  1\end{pmatrix}, \]
which has a singular Jacobian along both axes and a rank-deficiency of two at the origin. E.g. at $(x_1, 0)$, $\abs{x_1} \ll 1$, $x_1\neq 0$, the Jacobian has a singular value which is small but not zero. Consequently the pseudoinverse $J^*$ has a large norm and dominates $\frac{1}{\epsilon} H^*$ even for quite small $\epsilon$, which limits the neighborhood where the approximation of Theorem \ref{thmDirDeriv} is applicable.

In contrast, the Jacobian of,
\[ F: \mathbb{R}^2\rightarrow\mathbb{R}^2,x\mapsto -x_1 \begin{pmatrix}
1 & 7 \\ 8 & 3
\end{pmatrix}x - 10^6\begin{pmatrix}
1.1 \\  1\end{pmatrix} \]
has rank-deficiency two along the whole $x_2$-axis, but no actual crossing of singular manifolds (Figure \ref{figDS}).

\begin{figure}[h!tb]
    \begin{center}
		\begin{minipage}{0.9\linewidth}
		\includegraphics[width=\linewidth]{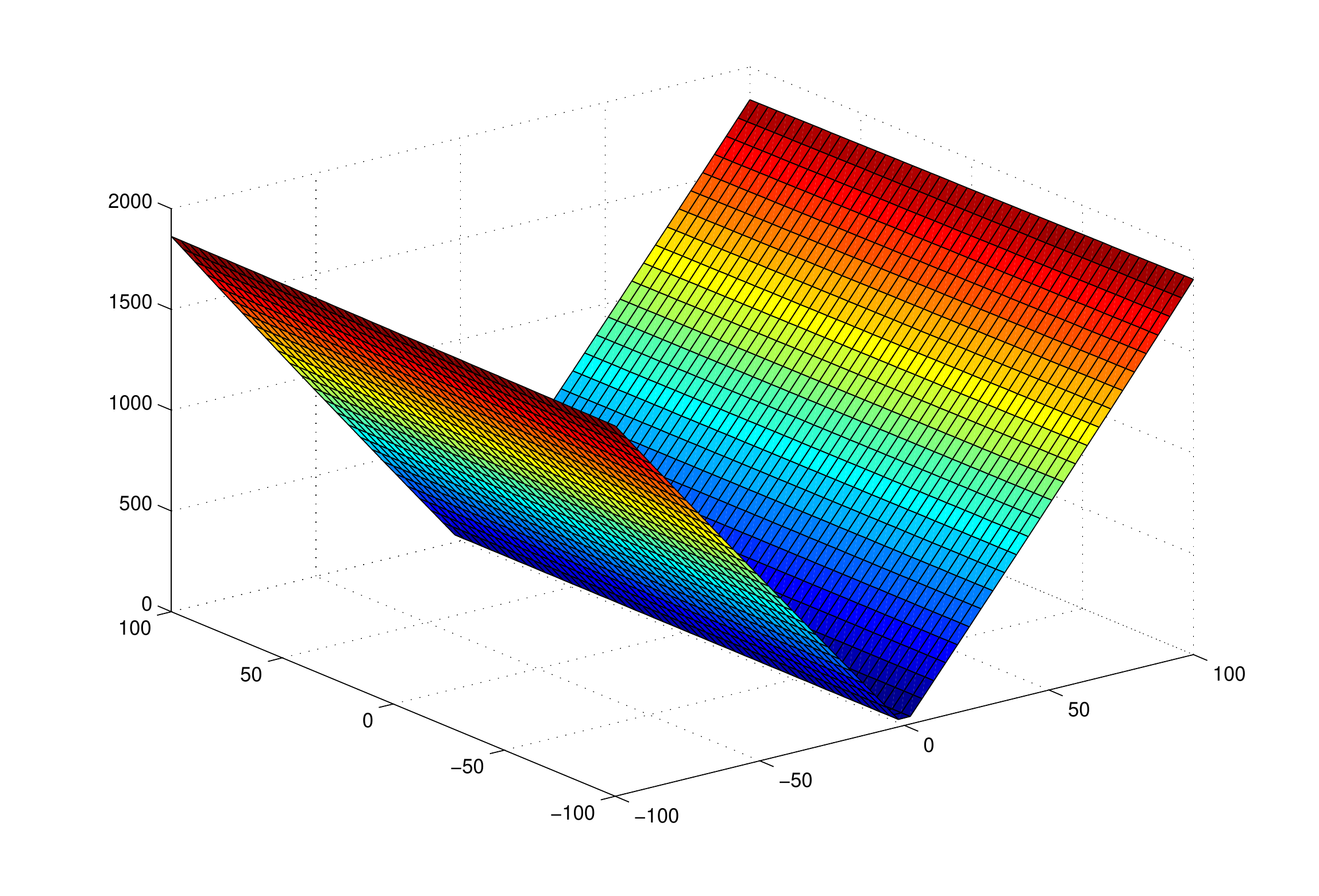}
		\caption[...]{$g$ near coinciding singular manifolds \label{figDS}}
		\end{minipage}
		\end{center}
\end{figure}

\section{Approximate stepsize control} \label{secApproximate}

We observe that \eqref{exstep} is similar to the stepsize control in Deuflhard's affine covariant globalization, which is
\begin{equation}\tag{AS}\label{approxstep}
\lambda^{(k)} \overset{!}{\leq} \norm{DF\left(x^{(k)}\right)^{-1}D^2F\left(x^{(k)}\right)\left(\Delta x^{(k)},\frac{\Delta x^{(k)}}{\norm{\Delta x^{(k)}}}\right)}^{-1}.
\end{equation}

Near singular $DF$ we find for any $v$ that $DF^{-1}v$, and hence both sides of the scalar product in \eqref{exstep}, will approximately be multiples of the left singular vector belonging to the smallest singular vector of $DF$, which suggests that \eqref{exstep} and \eqref{approxstep} coincide in the limit.

Indeed, we can show that this choice of stespsize works well if the smallest singular value of $DF$ is isolated and some genericity conditions hold: 

Let $DF(x) = U(x)\Sigma(x)V(x)^T$ be a smooth SVD of $DF(x)$ as defined in Appendix \ref{smoothSVD}, with $U = (u_1,\ldots,u_n)$, $V = (v_1,\ldots,v_n)$, $\Sigma = \text{diag}(\sigma_1,\ldots,\sigma_n)$. Then

\begin{align*}
DF^{-1}F & = \sum_{k=1}^{n} \frac{1}{\sigma_k} <u_k,F> v_k \\
 & = \frac{1}{\sigma_n}\left[ \left<u_n,F\right> v_n + O\left( \frac{\sigma_n}{\sigma_{n-1}} \norm{F} \right) \right].
\end{align*}

Similarly,
\begin{align*}
-\Delta x = DF^{-1}D^2F(\cdot,-\Delta x) & = \frac{1}{\sigma_n} \left[ \left<u_n,D^2F(\cdot,-\Delta x)\right>v_n + 
	O \left(\frac{\sigma_n}{\sigma_{n-1}}\norm{D^2F(\cdot,-\Delta x)}\right) \right]
\end{align*}
\begin{equation*}
\begin{array}{l}
DF^{-1}D^2F(-\Delta x,-\Delta x) = \\ = \frac{1}{\sigma_n} \left[ v_n u_n^T D^2F\left( \frac{1}{\sigma_n}\left[ \left<u_n,F\right> v_n + O\left( \frac{\sigma_n}{\sigma_{n-1}} \norm{F} \right) \right] ,-\Delta x\right) \right.
 \left. + O \left(\frac{\sigma_n}{\sigma_{n-1}}\norm{D^2F(-\Delta x,-\Delta x)}\right) \right]
\\ = \frac{1}{\sigma_n^2} \left[ \left<u_n,F\right>v_n u_n^T D^2F(v_n,-\Delta x) + O\left(\frac{\sigma_n}{\sigma_{n-1}}\norm{F}\norm{D^2F(\cdot,-\Delta x)}\right) \right] 
\\ \quad + \frac{1}{\sigma_n} O\left(\frac{\sigma_n}{\sigma_{n-1}}\norm{D^2F(-\Delta x, -\Delta x)}\right)
\\ = \frac{1}{\sigma_n^2}\left[ \left<u_n,F\right>\left<\nabla\sigma_n,-\Delta x\right>v_n+ O\left(\frac{\sigma_n}{\sigma_{n-1}}\norm{F}\norm{D^2F(\cdot,-\Delta x)}\right) \right] \\
\quad +\frac{1}{\sigma_n} O\left(\frac{\sigma_n}{\sigma_{n-1}}\norm{D^2F(-\Delta x,-\Delta x)}\right)
\end{array}
\end{equation*}

\begin{equation*}
\begin{array}{l}
\frac{\norm{\Delta x}}{\norm{DF^{-1}D^2F(-\Delta x,-\Delta x)}} =
\\ = \abs{\frac
{\frac{1}{\sigma_n}\left[ \left<u_n,F\right> + O\left(\frac{\sigma_n}{\sigma_{n-1}}\norm{F}\right) \right]}
{\frac{1}{\sigma_n^2}\left[ \left<u_n,F\right>\left<\nabla\sigma_n,-\Delta x\right> + O\left(\frac{\sigma_n}{\sigma_{n-1}}\norm{F}\norm{D^2F(\cdot,-\Delta x)}\right) \right]+\frac{1}{\sigma_n}O\left(\frac{\sigma_n}{\sigma_{n-1}}\norm{D^2F(-\Delta x,-\Delta x)}\right)}
}
\\ = \abs{\frac
{1+O\left(\frac{\sigma_n}{\sigma_{n-1}}\frac{\norm{F}}{\left<u_n,F\right>}\right)}
{\frac{1}{\sigma_n}\left[ \left<\nabla\sigma_n,-\Delta x\right>+ O\left(\frac{\sigma_n}{\sigma_{n-1}}\frac{\norm{F}}{\left<u_n,F\right>}\norm{D^2F(\cdot,-\Delta x)}\right) \right] + O\left(\frac{\sigma_n}{\sigma_{n-1}}\norm{D^2F(-\Delta x,-\Delta x)}\right)}
}
\end{array}
\end{equation*}
As $\norm{D^2F(-\Delta x,-\Delta x)} = O\left(\frac{1}{\sigma_n} \norm{F} \norm{D^2F(\cdot,- \Delta x)}\right)$ and $\left<\nabla\sigma_n,-\Delta x\right>$ $=$ $\left<u_n,D^2F(v_n,-\Delta x)\right>$, it follows that
\begin{equation} \label{apstepErr}
\frac{\norm{\Delta x}}{\norm{DF^{-1}D^2F(-\Delta x,-\Delta x)}} =
\abs{\frac{\sigma_n}{\left<\nabla\sigma_n,-\Delta x\right>}}
\left( 1+O\left(\frac{\sigma_n}{\sigma_{n-1}}\frac{\norm{F}}{\left<u_n,F\right>}\frac{\norm{D^2F(\cdot,-\Delta x)}}{\left<u_n,D^2F(v_n,-\Delta x)\right>}\right) \right),
\end{equation}
and hence for the stepsize
\begin{equation*}
\lambda^{(k)} \overset{!}{\leq} \abs{\frac{\sigma_n}{\left<\nabla\sigma_n,-\Delta x\right>}} \approx \norm{DF\left(x^{(k)}\right)^{-1}D^2F\left(x^{(k)}\right)\left(\Delta x^{(k)},\frac{\Delta x^{(k)}}{\norm{\Delta x^{(k)}}}\right)}^{-1}.
\end{equation*}

If $\frac{1}{\sigma_{n-1}}\frac{\norm{F}}{\left<u_n,F\right>}\frac{\norm{D^2F(\cdot,-\Delta x)}}{\left<u_n,D^2F(v_n,-\Delta x)\right>}$ is bounded in some suitable neighborhood, the error term in \eqref{apstepErr} reduces to $O(\sigma_n)$ and $\sigma_n$ converges quadratically to $0$ for  the iteration with stepsize $\eqref{approxstep}$.\footnote{As before assuming equality in \eqref{approxstep} and $\sigma_n / \left< \nabla\sigma_n, -\Delta x\right> < 0$.}

\section{Numerical Experiments}
We apply both stepsize controls to the test problem \emph{Expsin} from \cite{nw92}, which also appears as Example 3.2 in \cite{deu04}. The task is to find solutions of
\begin{equation} \label{expsin}
\begin{aligned}
\exp (x^2 + y^2)-3 & = 0,
\\ x + y - \sin \left(3(x+y)\right) & = 0.
\end{aligned}
\end{equation}
The Jacobian is singular along the lines $ y = x$ and
\[ y 0 -x\pm \frac{1}{3}\arccos\left( \frac{1}{3} \right) \pm \frac{2}{3} \pi. \]
\begin{figure}[h!tb]
		\begin{center}
		\begin{minipage}{0.47\linewidth}
		\includegraphics[height=\linewidth]{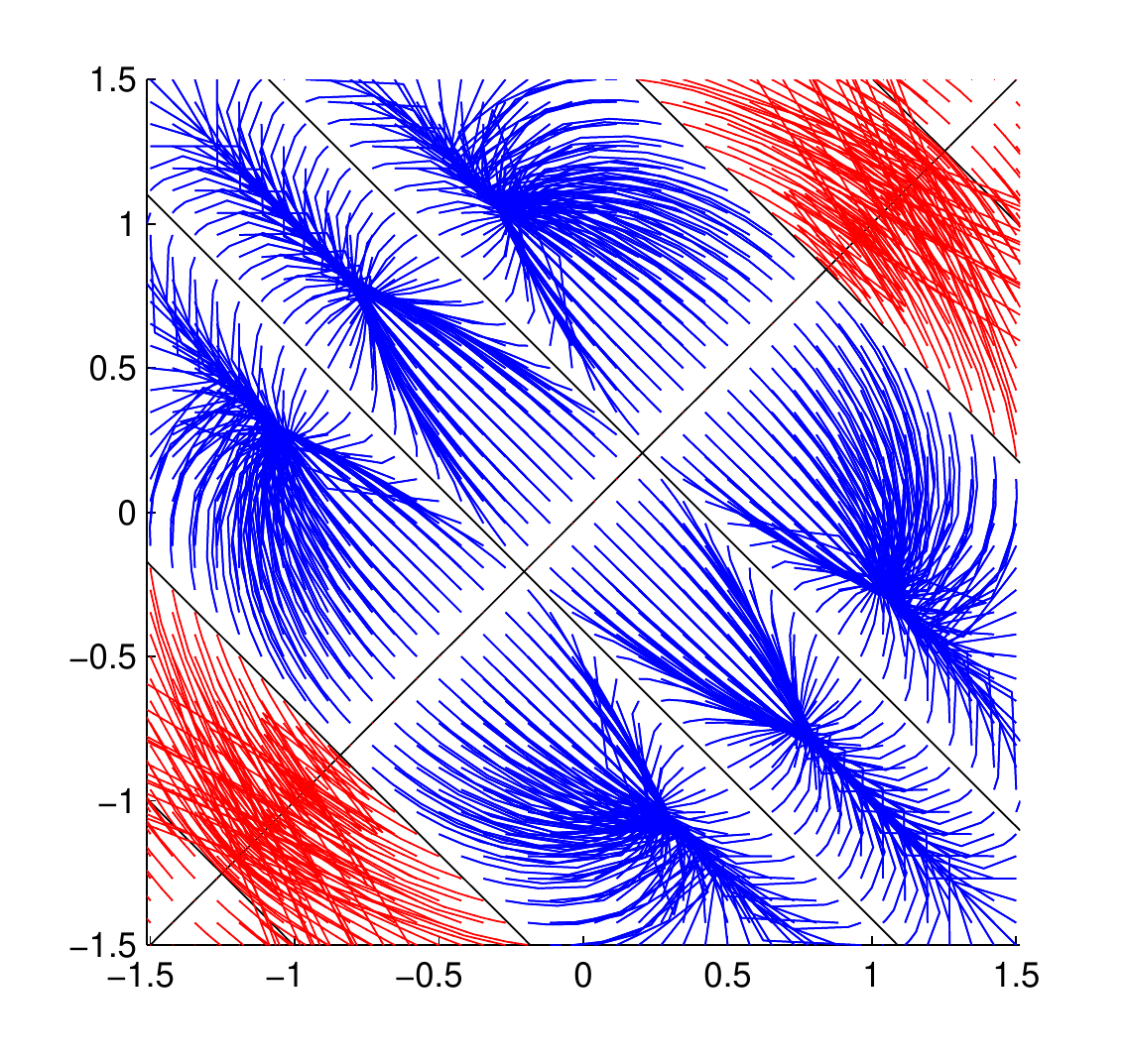}
		
		\end{minipage}
		\hfill
		\begin{minipage}{0.47\linewidth}
		\includegraphics[height=\linewidth]{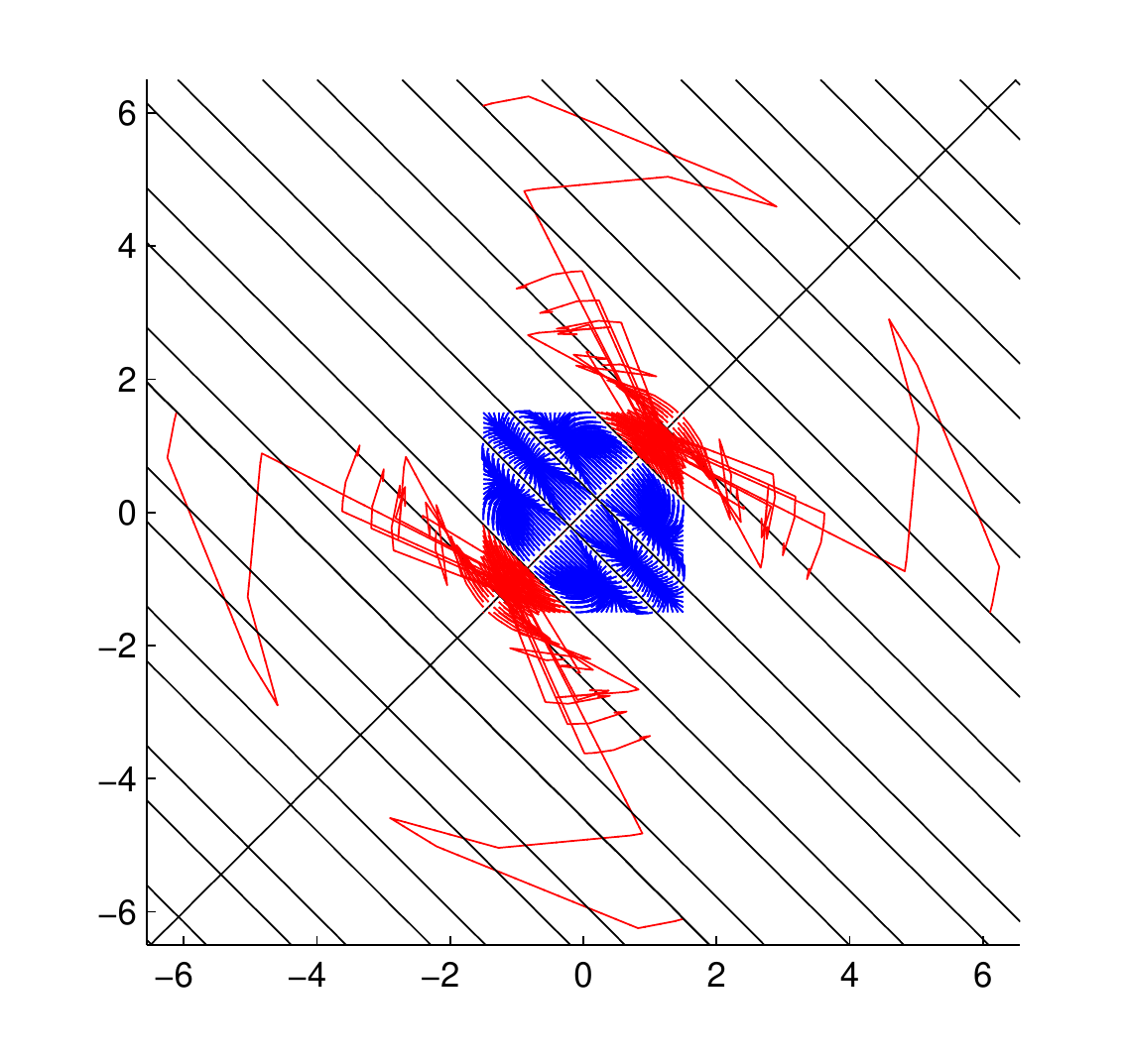}
		\end{minipage}
		\end{center}
		\caption{Iterates with stepsize control \eqref{exstep} for several initial values at two zoom levels. Blue paths converge to solutions, red paths converge to singular manifolds, black lines are singular manifolds \label{ESplot}}
\end{figure}

\begin{figure}[h!tb]
		\begin{center}
		\begin{minipage}{0.9\linewidth}
		\includegraphics[width=\linewidth]{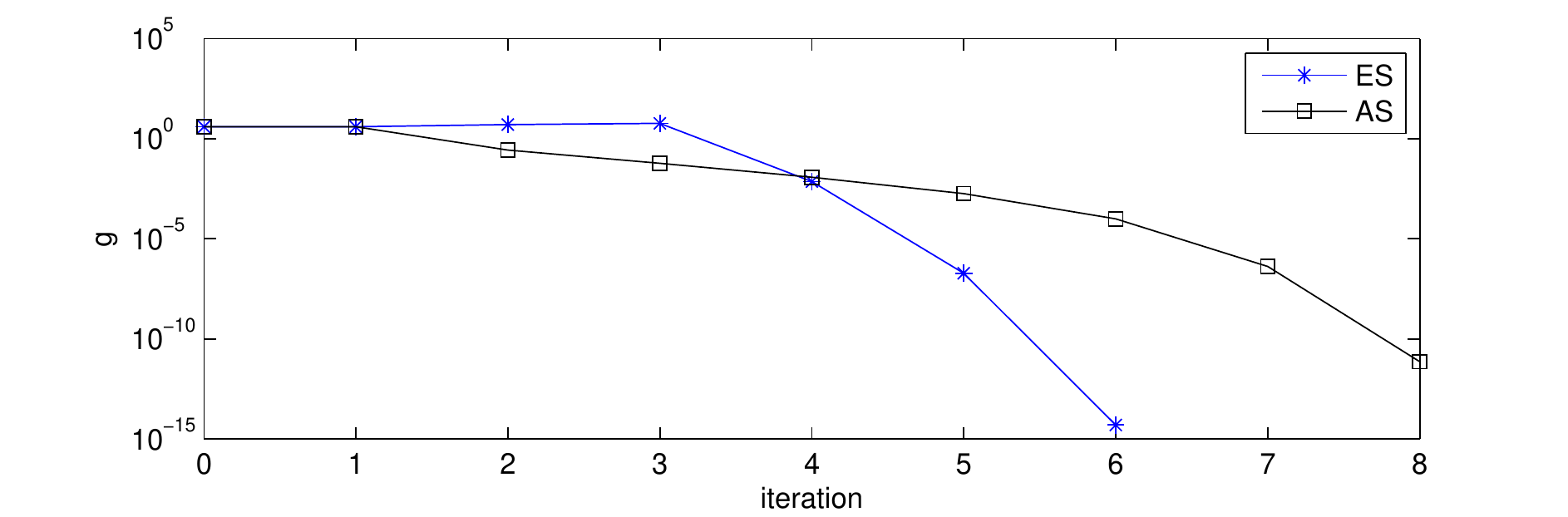}
		\caption{Convergence of $g$ for $(x_0,y_0) = (-0.5, -1.5)$ \label{gConv}}
		\end{minipage}
		\end{center}
\end{figure}

Figure \ref{ESplot} shows the path of Newton iterates with stepsize control \eqref{exstep} started at points 
in the square $[-1.5,1.5]^2$. Note that \eqref{expsin} has solutions only in the six middle regions and observe that iterates to initial points in regions without a solution eventually converge to a singular manifold while oscillating around it. Some of the paths do however cross several singular manifolds and end up at a "wrong" one. The approximate stepsize control appears more robust as the norm in \eqref{approxstep} does not ignore the components perpendicular to $T$. Indeed, all iterations converged to the correct solution or singular manifold when applying the approximate stepsize control to the same example, but the convergence to singular manifolds was slightly slower as exemplified in Figure \ref{gConv}.

As mentioned in Section \ref{secApproximate}, \eqref{exstep} and \ref{approxstep} are in general approximately equal close to singular manifolds, which suggests to compute both and use \eqref{exstep} only when they differ by less than some small factor. 
Due to shared terms this involves almost no additional effort. In our experiments this heuristic combined the faster convergence of \eqref{exstep} with the robustness of \eqref{approxstep}.

\appendix
\section{Connection to Griewank/Reddien}\label{grCon}
The indicator $g$ is inspired by the following Lemma of Griewank and Reddien, which gives an indicator for the rank-deficiency of rectangular matrices.
\begin{lem}
\label{gr_ind} (\cite{gr84}, Lemma 2.1) Let $D\subset \mathbb{R}^d$, $m\leq n$ and define $p:=n+1-m$. Let $A: D \rightarrow \mathbb{R}^{n \times m}$, $T: D \rightarrow \mathbb{R}^{n \times p}$ and $R: D \rightarrow \mathbb{R}^m$ be continuously differentiable functions so that
\begin{equation*}
\left(\begin{matrix}
A & -R \\
-T^T & 0
\end{matrix}\right)
\end{equation*}
is nonsingular for all $x\in D$. Then there are unique functions $u: D \rightarrow \mathbb{R}^m$, $V: D \rightarrow \mathbb{R}^{n \times p}$ and $g: D \rightarrow \mathbb{R}^p$ such that
\begin{equation*}
\begin{array}{ll}
AV = Rg^T, & \qquad T^T V = I, \\
u^TA = g^T T^T, & \qquad u^T R = 1,
\end{array}
\end{equation*}
and
\begin{equation*}
\operatorname{rank}(A(x)) = m-1 \text{ if and only if } g(x) = 0.
\end{equation*}

\end{lem}

\begin{lem}
For square matrices the indicator $g$ of Lemma \ref{gr_ind} coincides with \eqref{def_g}.
\end{lem}
\begin{proof}
As $n=m$, we have $p=1$, $g$ is a scalar and $T$, $R$ and $V$ are vectors. From $T^T V = I$ we obtain the decomposition $V = T + V^{(\perp)}$, where $V^{(\perp)} \in T^\perp := \operatorname{span} \left\{ T \right\}^\perp$. Projection of $gR = AV = AV^{(\perp)} + AT$ onto $(AT^\perp)^\perp$ yields $gP_{(AT^\perp)^\perp}R = P_{(AT^\perp)^\perp}AT$. As $(AT^\perp)^\perp = \operatorname{span} \left\{ A^{-T}T \right\}$, it follows that
\begin{equation*}
g = \frac{\left< A^{-T}T, AT \right>}{\left< A^{-T}T, R \right>} = \frac{\left<T, A^{-1}AT \right>}{\left< T, A^{-1}R \right>} = \frac{\norm{T}^2}{\left< T,A^{-1}R \right>}. \qedhere
\end{equation*}
\end{proof}

\section{Smooth SVD}\label{smoothSVD}
The singular value decomposition (SVD) of a smooth (i.e. sufficiently differentiable) family of matrices is only smooth if the singular values do not cross each other or $0$. However, a smooth SVD can sometimes be obtained if the requirements $\sigma_k \geq 0$, $\sigma_k > \sigma_{k+1}$ are given up. Such decompositions have already been studied extensively, see e.g. \cite{bg91}, \cite{wr92}. We will give here a concise proof for the existence of a differentiable SVD of a square matrix with isolated singular values, where one singular value crosses $0$, which is sufficient for our purposes.

Let $A = U\Sigma V$ be the SVD of $A\in\mathbb{R}^{n\times n}$ and $B: = A^T A$.
For an eigenpair $(\lambda, v)$ of $B$, consider the variation of its defining equations
\begin{align}
\label{eigEq} Bv & = \lambda v, \\
\label{eigNorm} \norm{v} & = 1.
\end{align}
We have from \eqref{eigEq}
\begin{align}
\nonumber (B+dB)(v + dv) & = (\lambda + d\lambda)(v + dv) \\
\label{eigVar} \Rightarrow dB\cdot v & = -(B-\lambda) dv + d\lambda\cdot v.
\end{align}
From \eqref{eigNorm} it follows that $v\perp dv$. As the eigenspace is one-dimensional, this implies $v\perp Bdv$ and hence \eqref{eigVar} can be decomposed into
\begin{align}
d\lambda & = \left<v, dB\cdot v\right> \\
\label{eigVarV} P_v^\perp dB\cdot v & = (\lambda - B) dv,
\end{align}
where $P_v^\perp$ is the orthogonal projector onto $\operatorname{span} \left\{ v \right\}^\perp$.
If $\lambda$ is an isolated eigenvalue of $B$, then $\mathcal{N}(B-\lambda)$$=$$\operatorname{span} \left\{ v \right\}$ and the unique solution of \eqref{eigVarV} under the constraint $v\perp dv$ is given by
\begin{equation}
dv = (\lambda - B)^+ dB\cdot v,
\end{equation}
where $^+$ denotes the Moore-Penrose pseudoinverse.

If $\lambda\neq 0$ we obtain smooth functions for the corresponding singular value $\sigma$ of $A$ and a left singular vector $u$ with $\norm{u} = 1$ by taking
\begin{align}
\sigma & = \sqrt{\lambda} \\
u & = \frac{Av}{\sigma}.
\end{align}

From $d\lambda$$=$$\left< v,dB\cdot v \right>$$=$$\left<v, (A^T dA+dA^T A)v\right>$$=$$2\left< Av, dA\cdot v\right>$$=$$\frac{2}{\sigma}\left< u, dA\cdot v\right>$ and $\sigma = \sqrt{\lambda}$ we obtain
\begin{equation}
d\sigma = \left<u, dA\cdot v\right>.
\end{equation}

If $\lambda = 0$, we have (using $Av = 0$)
\begin{align*}
(A+dA)(v+dv) & = Av + Adv + dA\cdot v \\
& = A(0-B)^+ dB\cdot v + dA\cdot v \\
& = -AB^+(dA^T A + A^T dA)v + dA\cdot v \\
& = (Id - A(A^T A)^+ A^T)dA\cdot v \\
& = P_{\mathcal{R}(A)^\perp}dA\cdot v \\
& \overset{!}{=} (\sigma + d\sigma)(u+ du) \\
& = d\sigma \cdot u.
\end{align*}
In this case, $u$ is given (up to a sign change) by $\mathcal{R}(A)^\perp = \operatorname{span} \left\{ u \right\}$ and we obtain again
\begin{align}
d\sigma & = \left<u, dA\cdot v \right>.
\end{align}

As $u$ is an eigenvector of $AA^T = B^T$, a calculation similar to \eqref{eigVar} yields (for all $\sigma$)
\begin{equation}
du = (\lambda - B^T)^+ dB^T\cdot u.
\end{equation}

\bibliographystyle{alpha}
\bibliography{newton}
\end{document}